\documentclass[12pt,reqno]{article}
\usepackage{amsmath,amsthm,amsfonts,amssymb,amscd,amsbsy}
\usepackage{graphics}

\usepackage{color}

\begin{document}

\theoremstyle{plain}
\newtheorem{thm}{Theorem}[section]
\newtheorem{theorem}[thm]{Theorem}
\newtheorem{main theorem}[thm]{Main Theorem}
\newtheorem{lemma}[thm]{Lemma}
\newtheorem{corollary}[thm]{Corollary}
\newtheorem{proposition}[thm]{Proposition}

\theoremstyle{definition}
\newtheorem{notation}[thm]{Notation}
\newtheorem{claim}[thm]{Claim}
\newtheorem{remark}[thm]{Remark}
\newtheorem{remarks}[thm]{Remarks}
\newtheorem{problem}[thm]{Problem}
\newtheorem{conjecture}[thm]{Conjecture}
\newtheorem{definition}[thm]{Definition}
\newtheorem{example}[thm]{Example}

\newcommand{\Max}{{\rm Max \ }}
\newcommand{\sA}{{\mathcal A}}
\newcommand{\sB}{{\mathcal B}}
\newcommand{\sC}{{\mathcal C}}
\newcommand{\sD}{{\mathcal D}}
\newcommand{\sE}{{\mathcal E}}
\newcommand{\sF}{{\mathcal F}}
\newcommand{\sG}{{\mathcal G}}
\newcommand{\sH}{{\mathcal H}}
\newcommand{\sI}{{\mathcal I}}
\newcommand{\sJ}{{\mathcal J}}
\newcommand{\sK}{{\mathcal K}}
\newcommand{\sL}{{\mathcal L}}
\newcommand{\sM}{{\mathcal M}}
\newcommand{\sN}{{\mathcal N}}
\newcommand{\sO}{{\mathcal O}}
\newcommand{\sP}{{\mathcal P}}
\newcommand{\sQ}{{\mathcal Q}}
\newcommand{\sR}{{\mathcal R}}
\newcommand{\sS}{{\mathcal S}}
\newcommand{\sT}{{\mathcal T}}
\newcommand{\sU}{{\mathcal U}}
\newcommand{\sV}{{\mathcal V}}
\newcommand{\sW}{{\mathcal W}}
\newcommand{\sX}{{\mathcal X}}
\newcommand{\sY}{{\mathcal Y}}
\newcommand{\sZ}{{\mathcal Z}}
\newcommand{\A}{{\mathbb A}}
\newcommand{\B}{{\mathbb B}}
\newcommand{\C}{{\mathbb C}}
\newcommand{\D}{{\mathbb D}}
\newcommand{\E}{{\mathbb E}}
\newcommand{\F}{{\mathbb F}}
\newcommand{\G}{{\mathbb G}}
\newcommand{\HH}{{\mathbb H}}
\newcommand{\I}{{\mathbb I}}
\newcommand{\J}{{\mathbb J}}
\newcommand{\M}{{\mathbb M}}
\newcommand{\N}{{\mathbb N}}
\renewcommand{\P}{{\mathbb P}}
\newcommand{\Q}{{\mathbb Q}}
\newcommand{\re}{{\mathbb R}}
\newcommand{\R}{{\mathbb R}}
\newcommand{\T}{{\mathbb T}}
\newcommand{\U}{{\mathbb U}}
\newcommand{\V}{{\mathbb V}}
\newcommand{\W}{{\mathbb W}}
\newcommand{\X}{{\mathbb X}}
\newcommand{\Y}{{\mathbb Y}}
\newcommand{\Z}{{\mathbb Z}}
\newcommand{\la}{{\lambda}}
\newcommand{\al}{{\alpha}}
\newcommand{\be}{{\beta}}

\newcommand{\ve}{{\varepsilon}}
\newcommand{\vr}{{\varphi}}

\def\ga{{\gamma}}
\def\vr{{\varphi}}
\def\la{{\lambda}}
\def\al{{\alpha}}
\def\be{{\beta}}
\def\ve{{\varepsilon}}
\def\lv{\left\vert}
\def\rv{\right\vert}

\def\sen{\operatorname{{sen}}}
\def\tr{\operatorname{{tr}}}

\def\re{{\Bbb{R}}}
\def\bc{{\mathbb C }}
\def\fT{{\frak T}}
\def\nb{{\mathbb N}}
\def\bz{{\mathbb Z}}

\def\pc{{\mathcal P}}

\centerline{\huge\bf Complexity and fractal dimensions for}

\medskip

\centerline{\huge\bf infinite sequences with positive entropy}

\vskip .3in

\centerline{\sc Christian Mauduit}

\centerline{\sl Universit\'e d'Aix-Marseille et Institut Universitaire de France}

\centerline{\sl Institut de Math\'ematiques de Marseille, UMR 7373 CNRS,}

\centerline{\sl 163, avenue de Luminy, 13288 Marseille Cedex 9, France}


\centerline{\sc Carlos Gustavo Moreira}

\centerline{\sl Instituto de Matem\'atica Pura e Aplicada}

\centerline{\sl Estrada Dona Castorina 110}

\centerline{\sl 22460-320 Rio de Janeiro, RJ, Brasil}

\vskip .3in

{\bf Abstract:}
The complexity function of an infinite word $w$ on a finite alphabet $A$ is the sequence counting, for each non-negative $n$, the number of words of length $n$ on the alphabet $A$ that are factors of the infinite word $w$.
The goal of this work is to estimate the number of words of length $n$ on the alphabet $A$ that are factors of an infinite word $w$ with a complexity function bounded by a given function $f$ with exponential growth and to
describe the combinatorial structure of such sets of infinite words. We introduce a real parameter, the {\it word entropy} $E_W(f)$ associated to a given function $f$
 and we determine the fractal dimensions of sets of infinite sequences with complexity function bounded by $f$ in terms of its word entropy.
We present a combinatorial proof of the fact that $E_W(f)$ is equal to the topological entropy of the subshift of infinite words whose complexity is bounded by $f$ 
and we give several examples showing that even under strong conditions on $f$, the word entropy $E_W(f)$ can be strictly smaller than
the  limiting lower exponential growth rate of $f$.

\vskip .1in



This work was supported by CNPq, FAPERJ and the Agence Nationale de la Recherche project ANR-14- CE34-0009 MUDERA.

\vskip .3in

\section{Notations}

We denote
by $q$ a fixed integer greater or equal to $2$, by $A$ the finite
alphabet $A=\{0,1,\dots,q-1\}$, by $A^*=\bigcup\limits_{n\ge0} A^n$
the set of finite words on the alphabet $A$ and by $A^{\N}$ the set
of infinite words (or infinite sequences of letters) on the alphabet
$A$.
If $v\in A^n, n \in \N$ we denote $|v| = n$ the length of the word $v$ and if $S$ is a finite set, we denote by $|S|$ the number of elements of $S$.

If $w\in A^{\N}$ we denote by $L(w)$ the set of finite factors of $w$:
$$
L(w)=\{v\in A^*,\,\, \exists \, (v',v'')\in A^*\times A^{\N}, \, w=v'v v''\}
$$
and, for any non-negative integer $n$, we write $L_n(w)=L(w)\cap A^n$.
For any $v \in L(w)$, we denote by $d^+(v)$ the number of different right extensions of $v$ in $w$: $$d^+(v) = |\{a \in A, va \in L(w) \}|$$ and we say that $v$ is a special factor
of $w$ if $d^+(v) \ge 2$ (see [CN10] for a detailed study of these notions).
If  $x$ is a real number, we denote
$
\lfloor x\rfloor = \max\{n\in\Z, n\le x\},
\lceil x\rceil=\min\{n\in \Z, x\le n\}
$
and
$\{ x \} = x - \lfloor x\rfloor$.

We will use at several stage the following classical lemma concerning sub-additive sequences due to Fekete [Fek23]:

\begin{lemma}\label{lemFekete}
If $(a_n)_{n\ge1}$ is a sequence of real numbers such that $a_{n+n'} \le a_n + a_{n' }$ for any positive integers $n$ and $n'$, then the sequence
$\left( \frac {a_n} n\right)_{n\ge1}$ converges to $\inf_{n\ge 1}\frac {a_n} n$.
\end{lemma}

\begin{definition}\label{def1.2}
  The { \it complexity function} of $w\in A^{\N}$ is defined for any non-negative integer
  $n$ by $p_w(n)=|L_n(w)|$.

\end{definition}

\begin{example} \label{example Ch}
If $C_{\{0,1\}} = 01101110010111011110001001101010111100110111101111\dots$ is the Champernowne inifinite word on the alphabet $\{0,1\}$
obtained by concatenating the representation in base $2$ of the consecutive non-negative integers, we have $p_{C_{\{0,1\}}}(n) = 2^n$ for any non-negative integer
  $n$ (see [Ch33] and [MS98] for stronger results concerning statistical properties of Champernowne words)).
\end{example}
 
\noindent
For any $w\in A^{\N}$ and for any $(n,n')\in \N^2$ we have
  $L_{n+n'}(w)\subset L_n(w) L_{n'}(w)$ so that 
  \begin{equation} \label {p_w}
  p_w(n+n')\le p_w(n) p_w(n')
  \end{equation}
and it follows then from 
Lemmas \ref{lemFekete}  applied to $a_n = \log p_w(n)$
that for any
$w\in A^{\N}$, the sequence $\left( \frac 1n \log p_w(n)\right)_{n\ge1}$ converges
 to $\inf_{n\ge 1} \frac 1n \log p_w(n)$.
In particular, let us mention the following useful consequence:

\begin{claim}\label{consequence}
If there exist a real number $\lambda_0$ ($1 \leq \lambda_0 \leq q$) and an integer $n_0$ such that $p_w(n_0)<\lambda_0^{n_0}$,
then $p_w(n)=O(\lambda^n)$ for some $\lambda < \lambda_0$.
\end{claim}


We denote
$$E(w)=\lim\limits_{n\to\infty}\frac 1n\log p_w(n) = h_{top} (X(w), T)$$
the topological entropy of the symbolic dynamical system $(X(w),T)$ where
$T$ is the one-sided shift on $A^{\N}$ and $X=\overline{orb_T(w)}$ is
the closure of the orbit of $w$ under the action of $T$ in $A^{\N}$
(see for example [Fer99] or [PF02] for a detailed study of the notions of complexity function and topological entropy).

\section{Presentation of the results}







Our work concerns the study of infinite sequences $w$ the complexity
function of which is bounded by a given function $f$ from $\N$ to $\R^{+}$.
More precisely, if $f$ is such a function, we consider the set
$$W(f)=\{w\in A^{\N}, p_w(n)\le f(n), \forall n \in \N\}$$
and we denote
$$\sL_n(f)=\bigcup\limits_{w\in W(f)} L_n(w).$$ 
For any $(n,n')\in \N^2$ we have
$\sL_{n+n'}(f)\subset \sL_n(f) \sL_{n'}(f)$ so that we can deduce from Lemma \ref{lemFekete} applied to $a_n = \log |\sL_n(f)|$ that the sequence
$\left( \frac 1n \log |\sL_n(f)| \right)_{n\ge1}$ converges
 to \break $\inf_{n\ge 1} \frac 1n \log |\sL_n(f)|$,
 which is the topological entropy of the subshift $(W(f), T)$ :
 $$h_{top} (W(f), T) = \lim_{n\to+\infty}\frac 1n\log |\sL_n(f)| = \inf_{n\ge 1} \frac 1n \log |\sL_n(f)|.$$
We denote by $E_0(f)$ the limiting lower exponential growth rate of $f$
$$E_0(f)=\lim\limits_{n\to\infty} \inf \frac 1n \log f(n).$$
Our previous papers [MM10] and [MM12] concern the case $E_0(f)=0$ and in this paper we will consider the case of positive entropy, for which very few results are known since the
work of Grillenberger [Gri73].

We define in this work the notions of w-entropy (or word-entropy) and w-complexity (or word complexity) of $f$ as follow :

\begin{definition}\label{defg}
If $f$ is a function from $\N$ to $\R^{+}$, the
{\it w-entropy} (or  {\it word entropy}) of $f$ is the quantity
$$E_W (f) =\sup_{\substack{w \in W(f)}} E(w)$$
and  {\it w-complexity} (or  {\it word complexity}) of $f$ the function $P_f$ defined, for any $n\in\nb$ by
$$P_f(n)=\max\{p_w(n), w\in W(f)\}.$$
\end{definition}

\begin{remark}
It will follow from Theorem \ref {prop3.0} that $E_W(f)$ and $E_0(P_f)$ are both equal to the topological entropy of the subshift
$(W(f), T)$.
\end{remark}

In Section \ref{4.1} we prove an useful general lemma on the positivity of the Hausdorff measure corresponding to the Hausdorff dimension of an invariant compact set 
by an expanding dynamical system on an interval and in
Section \ref{firsttheorem} we give a combinatorial proof of the following :

\begin{theorem}\label{prop3.0}
For any function $f$ from $\N$ to $\R^{+}$, we have
$$E_W(f) = \lim_{n\to+\infty}\frac 1n\log (|\sL_n(f)|) = h_{top} (W(f), T).$$
\end{theorem}

\begin{remark} \label {theorem}
As mentioned in the introduction of Section \ref {firsttheorem}, Theorem \ref{prop3.0} can be obtained as a consequence of the variational principle.
But we will present a direct proof of Theorem \ref {prop3.0}.
This proof, which avoids the use of the variational principle, introduces useful quantitative tools that can be used for other applications.

\end{remark}

To each $x \in  [0,1]$ we associate the infinite word
\begin{equation} \label {w}
\bold{w}(x) = {w_0}{w_1}\cdots{w_i}\cdots
\end{equation}
on the alphabet $A$,
where $x = \sum \limits_{i \ge 0} \frac{w_i}{q^{i+ 1}}$ is the
representation in base $q$ of the real number $x$ (when $x$ is a $q$-adic
rational number, we choose for $x$ the infinite word ending with $0^\infty$).
We show in Section \ref{application} that the w-entropy of $f$ allows to compute exactly the fractal dimensions of the set
\begin{equation} \label {C(f)}
C(f)=\{x = \sum \limits_{i \ge 0} \frac{w_i}{q^{i+ 1}} \in [0,1] , \bold{w}(x)  = {w_0}{w_1}\cdots{w_i}\cdots \in W(f)\}
\end {equation}
of real numbers $x \in [0,1]$ the $q-$adic expansion of which has a
complexity function bounded by $f$ (Theorem \ref {theoremC}).
The reader will find in [Fal90, Chapters 2 and 3] basic definitions concerning fractal dimensions and in [Bug04, Chapters V and VI] 
other examples of application of fractal dimensions to number theoretical problems.

In Section \ref {?} we show that for any function $f$ from $\N$ to $\R^{+}$ such that $f(n) \ge n+1$ for any non-negative integer $n$ and $E_0(f)>0$
we have $E_W(f) > 0$.
But we will see that it is easy to give examples of function $f$ for which the ratio $E_W(f)/E_0(f)$ (which is always
smaller or equal to $1$) can be made arbitrarily small
and we will show in Section \ref {relations} that, even under stronger conditions concerning the function $f$, it might happen that $E_W(f) < E_0(f)$.
On the other hand, if $f$ is indeed a complexity function (i.e. $f=p_w$ for
some $ w \in A^\N$), then we clearly have $E_W (f)=E_0 (f)$.
But the following problem seems difficult to handle:

\begin{problem}
Find a set of simple conditions on $f$ which hold for complexity functions and implies $E_W(f)=E_0(f)$. 
\end{problem}

 In section 8 we show that,
for any $h \in (0,+\infty)$, there is an infinite word $w$ 
on a finite alphabet with at least $\lceil e^h \rceil$ letters such that
the complexity function $p_w$ has exact exponential order $h$ (in the sense that there exist positive constants $c_1$ and $c_2$ such that, 
for any $n \in \N$,
we have $c_1 e^{hn} \le p_w(n) \le c_2 e^{hn}$).
Theorem \ref {6.10} and Corollary \ref {corollary} can be compared with the result obtained by Grillenberger in [Gri73] showing
that, for any $0 < h < \log q$, there exists a strictly ergodic $w\in A^\N$ such that $E(w) = \lim_{n\to+\infty}\frac 1n\log p_w(n) = h$.

In Section 9 we discuss an open problem concerning the w-entropy of the minimum of two functions.

\begin{remark}
It might be difficult to compute the w-entropy of a given function and it would be interesting to provide an algorithm allowing to estimate with 
an arbitrary precision
the w-entropy of any given function satisfying the conditions $(\mathcal C^*)$ (see Definition \ref {defC^*}).
\end{remark}





\section{A lemma concerning fractal dimensions}  \label{4.1}
 
\begin{definition}
We call {\it block interval} any interval of a decomposition of $[0,1]$ as the union of $q^n$ intervals of size $1/q^n$, for some $n \in \N$
(block intervals correspond to cylinders, i.e., to fixing the beginning of infinite sequence).
\end{definition}

Let $\psi$ be the shift map on the interval $[0, 1]$, defined by $\psi(x):=\{qx\}=qx-\lfloor qx \rfloor$.
Then, we have the following lemma concerning the fractal dimensions of $\psi$-invariant compacts of $[0, 1]$:

\begin{lemma}\label{lemdim}
If $X \subset [0, 1]$ is a compact set such that $\psi (X) \subset X$ with upper box dimension equal to $s$, the the Hausdorff measure of dimension $s$
of $X$ is greater or equal to $q^{-s}/2$ (or infinite). In particular, the upper box and Hausdorff dimensions of $X$ coincides.
\end{lemma}

\begin{proof}

If $X$ had Hausdorff measure of dimension $s$ smaller than $q^{-s}/2$, since $X$ is compact, it would be possible to cover $X$
by a finite number of open intervals $\tilde I_j (1 \le j \le m)$ with $$\sum_{j=1}^m |\tilde I_j|^{s} < q^{-s}/2.$$
Each interval $\tilde I_j$ can be covered by two block intervals with disjoint
interiors, say $I_{2j-1}$ and $I_{2j}$ with $|I_{2j-1}|, |I_{2j}|<q|\tilde I_j|$, so it is possible to cover $X$ by the block intervals with disjoint
interiors $I_j (1 \le j \le 2m)$ of sizes $1/q^{r_j} (1 \le j \le 2m)$ with $$\sum_{j=1}^{2m} (1/q^{r_j})^{s} < 2\cdot q^s\cdot q^{-s}/2=1.$$
This would imply that $X$ is contained in the maximal invariant set $K$ of the expanding map
$\varphi: \cup_{j=1}^m I_j \rightarrow [0,1]$ given by $\varphi \mid_{I_j}:=\psi^{r_j}$.
But the upper box and the Hausdorff dimensions of $K$ are both equal to $\lambda$, where $\lambda$ is the root of the equation
$$\sum_{j=1}^m (1/q^{r_j})^\lambda=1$$
(see [PT92, Chapter 4], in which upper box dimension is called limit capacity) and, since
$\sum_{j=1}^m (1/q^{r_j})^{s} < 1$, we would have
$\lambda < s$, which would contradict our hypothesis.
 
\end{proof}

The next corollary follows from the proof of Lemma \ref {lemdim}:
\begin{corollary}
For any $x \in [0, 1]$, the closure of the orbit of $x$ under the shift map $\psi$ has positive Hausdorff measure of dimension $E(\bold{w}(x) )/\log q$. 
\end{corollary}

\section{Word entropy and topological entropy} \label{firsttheorem}


Theorem \ref{prop3.0} can be obtained as a consequence of the variational principle.
To show it, we refer to the presentation given by P. Walters in [Wal82, Chapter 8].

By definition, $E_W(f)$ is the supremum of the topological entropy of all infinite words belonging to $W(f)$
(in other words, it is the supremum of the topological entropy of all the transitive components of $(W(f), T)$).
As T is an expansive map, it follows from Theorem 8.2 of [Wal82] that its entropy map is upper semicontinuous.
Then, it follows from Theorem 8.7, (v) of [Wal82] that there exists a measure of maximal entropy for T
and from Theorem 8,7, (iii) of [Wal82] that there is an ergodic such measure.
By considering a generic orbit with respect to this ergodic measure, we can get the conclusion.

But we present in this section a direct proof of Theorem \ref {prop3.0}, which
will follow 3 steps given in Sections 4.1, 4.2 and 4.3.


\subsection{Upper bounds for $|\sL_n(f)|$}

To obtain upper bounds for $|\sL_n(f)|$ we follow the strategy presented in [MM10, Section 3] which show that,
for any non-negative integers $k$ and  $N$, we have
$  |\sL_{kN}(f)| \le 
f(N)^k q^{N f(N)}.$
Then we choose the parameter $k$ in order to optimize this upper bound.
Let us fix an increasing sequence of positive integers $(N_r)_{r \in \N}$ such that
$$\lim_{r\to\infty} \log f(N_r)/N_r=E_0(f)$$ and, for $q^{q^{N_r}}\le n<q^{q^{N_{r+1}}}$, take $k=\lceil\frac n{N_r}\rceil$ in the above 
estimate  \footnote{  Clearly, we can replace in this estimate the function $q^{q^{N_r}}$
  by other unbounded increasing functions, with arbitrarily fast
  growth at infinity.}.
It follows that
$$
|\sL_n(f)|\le|\sL_{kN_r}(f)|\le f(N_r)^k q^{N_r f(N_r)}= e^{\vr(n)},
$$
where
$\vr(n)=k\log f(N_r)+ N_r f(N_r) \log q$
satisfies
$$
\lim_{n\to\infty}\frac 1n\vr(n)=\lim_{r\to\infty}\frac1{N_r}\log f(N_r)=E_0(f).
$$


Finally, we have proved the following:

\begin{claim}\label{prop2.1}
For any function $f$, we have
 $$|\sL_n(f)|\le \exp(n E_0(f)+ o(n)).$$

\end{claim}

\begin{example} \label{example C}
For each $1<\theta\le q$, and $n_0 \in \N$ such that $\theta^{n_0+1} >  n_0+q-1$, we define the function $f$ by $f(1)=q$,
$f(n) = n+q-1$ for $1 \le n \le n_0$ and $f(n)=\theta^n$ for $n > n_0$. We have $E_0(f) =  \log \theta$
and
$ |\sL_n(f)|\le \theta^{n + O((\log_q n)^{\log_q \theta} \log_q\log_q n)}.$

\end{example}

\subsection{A lemma which implies that $E_W(f) = E_0(P_f)$}


If $f$ is a function from $\N$ to $\R^{+}$, let $P_f$ be the w-complexity of $f$ defined in Definition \ref {defg}.
It is easy to check that for any $(n,n')\in\nb^2$ we have $P_f(n+n')\le P_f(n)P_f(n')$ and it follows from Lemma \ref{lemFekete} applied to
$a_n = \log P_f(n)$ that the sequence
$\left( \frac 1n {\log P_f(n)} \right)_{n\ge1}$ converges to $\inf_{n\ge 1} \frac 1n {\log P_f(n)}$,
so that we have $$E_0(P_f) = \lim_{n\to\infty}\frac 1n {\log P_f(n)} = \inf_{n\ge 1}\frac 1n {\log P_f(n)}.$$

\begin{lemma}\label{lem}
For any function $f$ from $\N$ to $\R^{+}$, there exists
$w\in W(f)$ such that for any $n\in\nb$ we have $p_w(n) \ge \exp (E_0(P_f) n)$. 
\end{lemma}

\begin{proof}

Let us prove Lemma \ref {lem} by contradiction and suppose that, for any $w \in W(f)$, there exists a positive integer $n$ such that $p_w(n) < \exp (E_0(P_f) n)$.
By Claim \ref {consequence}, this would imply that, for any $w \in W(f)$
$$E(w) = \lim_{n\to+\infty}\frac 1n\log p_w(n) < E_0(P_f).$$
Let us choose $\tilde m \in \N$ such that
$E(w)  < \frac {\tilde m} {\tilde m +1} E_0(P_f)$
and $$M= \sup_{n \in \N} p_w(n) \exp {(- \frac {\tilde m} {\tilde m +1} E_0(P_f) n)} < +\infty.$$
It follows that for $m = \max (\tilde m, M)$ we would have $$p_w(n) \leq m.\exp(\frac m{m+1}E_0(P_f) n),$$ so that we would have
$$W(f)\subset \cup_{m \ge 1}W(g_m),$$ where $g_m$ is defined by $g_m(n)=m.\exp(\frac m{m+1}E_0(P_f) n)$ for any $n \in \N$.
This would imply that the set $C(f)$ defined by (\ref {C(f)}) would satisfy $C(f) \subset \cup_{m \ge 1}C(g_m)$.

By Claim \ref{prop2.1}, we have
$$\lim_{n\to+\infty}\frac 1n\log (|\sL_n(g_m)|) \le E_0(g_m) = \frac m{m+1}E_0(P_f)<E_0(P_f),$$
so that the upper box dimension of $C(g_m)$ would be at most $$\frac m{m+1}c/\log q<c/\log q$$ since, for each $n \in \N$, $C(g_m)$
could be covered by $|\sL_n(g_m)|$ intervals of size $1/q^n$.
In particular, the Hausdorff measure of dimension $E_0(P_f)/\log q$ of $C(g_m)$ would be $0$, and thus the Hausdorff measure of dimension
$E_0(P_f)/\log q$ of $C(f) \subset \cup_{m \ge 1}C(g_m)$ would be also $0$.
By applying Lemma \ref {lemdim} to $X = C(f)$, this would imply that the upper box dimension of $C(f)$ is strictly smaller than $E_0(P_f)/\log q$,
which would be a contradiction.
\end{proof}

\begin{corollary}\label{theoremc}
For any function $f$ from $\N$ to $\R^{+}$, we have $E_W(f) = E_0(P_f)$.
\end{corollary}

\begin{proof}

It follows from Lemma \ref{lem} that there exists
$w\in W(f)$ such that for any integer $n\ge 1$ we have $\frac 1n\log p_w(n)  \ge E_0(P_f)$, so that $$E_W (f) =\sup_{\substack{w \in W(f)}} E(w) \ge E_0(P_f).$$
On the other hand, for any $\ve>0$, there exists $w \in W(f)$ such that 
$$\inf_{n\ge 1}\frac
1n\log p_w(n)=\lim_{n\to+\infty}\frac 1n\log p_w(n) > E_W(f)-\ve,$$
so that for any integer $n \ge 1$
we have $p_w(n) > \exp((E_W(f)-\ve)n)$.
This implies
that, for any $\ve>0$ and any $n\in\N$, we have $$P_f(n)=\max\{p_w(n), w\in W(f)\} > \exp((E_W(f)-\ve) n)$$
so that we have $P_f(n) \ge \exp(E_W(f) n)$ for any $n \in \N$.
In particular, $$E_0(P_f)=\lim_{n\to\infty}\frac{\log P_f(n)}n \ge E_W(f)$$ and it follows that $E_W(f)=E_0(P_f)$.

\end{proof}

\subsection{Lower bound for $|\sL_n(f)|$ and end of proof of Theorem \ref {prop3.0}}

It follows from Lemma \ref{lem} and Corollary \ref{theoremc} that there exists $w \in W(f)$ such that, for any $n\in\nb$, we have
\begin{equation} \label {gag}
p_w(n) \ge \exp (E_0(P_f) n) = \exp(E_W(f) n).
\end{equation}
As we have $|\sL_n(f)| \ge p_w(n)$ \text{for any} $n \in \N$, it follows from (\ref {gag}) that, for any $n \in \N$, we have
\begin{equation} \label {lower}
|\sL_n(f)| \ge \exp(E_W(f) n).
\end{equation}
On the other hand, for any $w\in W(f)$ and any $n \in \N$, we have $$p_w(n) \le \max\{p_w(n), w\in W(f)\}=P_f(n) \le f(n),$$
which implies that $W(f)$ coincides with $W(P_f)$. This is enough to end the proof of Theorem \ref{prop3.0}
since, by (\ref {lower}), Claim \ref{prop2.1}  and Corollary \ref{theoremc}  we have
$$E_W(f) 
\le \lim_{n\to+\infty}\frac 1n\log (|\sL_n(f)|)
 = \lim_{n\to+\infty}\frac 1n\log (|\sL_n(P_f)|) \le E_0(P_f)= E_W(f).$$

\vskip .1in



\section{Application to fractal dimensions of $C(f)$} \label {application}

Given a function $f$ from $\N$ to $\R^{+}$, let $C(f)$ be the set defined by (\ref{C(f)})
of real numbers $x \in [0,1]$ the $q-$adic expansion of which has a
complexity function bounded by $f$.

\begin{theorem} \label {theoremC}
$\,$ \hfill \break
The Hausdorff dimension of $C(f)$ is equal to $E_W(f)/\log q$. More precisely, the upper box dimension of $C(f)$ is equal to $E_W(f)/\log q$
and the Hausdorff measure of dimension $E_W(f)/\log q$ of $C(f)$ is positive (or infinite).
Moreover, there exists $ x \in C(f)$ such that the closure of the orbit of $x$ by the map $\psi$ has positive (or infinite) Hausdorff measure of dimension $E_W(f)/\log q$. 

\end{theorem}

\begin{proof}
The fact that the upper box dimension of $C(f)$ is at most $E_W(f)/\log q$ follows from the estimate $\lim_{n\to+\infty}\frac 1n\log (|\sL_n(f)|) = E_W(f)$ since, for any $n \in \nb$, $C(f)$ can be covered by $|\sL_n(f)|$ intervals of size $1/q^n$.

On the other hand, let us consider $w \in W(f)$ such that $p_w(n) \ge \exp(E_W(f) n)$ for any $n \in \N$ and $x \in [0,1]$ such that $\bold{w}(x)  = w$.
If $X$ is the closure of the orbit of $x$, $X$ is invariant by $\psi$ and intersects at least $\exp(E_W(f) n)$ intervals of the usual decomposition of $[0,1]$ as the union of $q^n$ intervals of size $1/q^n$, for each $n \in \N$.
It follows that the upper box dimension of $X$ is at least $E_W(f)/\log q$ and finally is equal to $E_W(f)/\log q$.
Theorem \ref {theoremC} follows by applying Lemma \ref {lemdim} to $X = C(f)$.

\end{proof}

\section{Exponential growth rate and word entropy}  \label{?}

It follows from the definitions that for any function $f$ from $\N$ to $\R^{+}$ we have $E_W (f)\le E_0 (f)$ and the next proposition shows that if we have $E_0 (f)>0$, then $E_W (f)>0$
(under the trivial necessary condition that $f(n) \ge n+1$ for any non-negative integer $n$).

\begin{proposition} \label{corollary2}
If $f$ is a function from $\N$ to $\R^{+}$ such that $f(n) \ge n+1$ for any $ n \in \N$ and $E_0(f)>0$, then $E_W(f) > 0$.
\end{proposition}

\begin{proof}
If $f$ is a function from $\N$ to $\R^{+}$ such that $E_0(f) > 0$, then there exists $c>0$ such that for any positive integer $n$ we have $f(n) > e^{cn}$.
If we consider the function $g$ from $\N$ to $\R^{+}$ defined by $g(n) = \max \{n+1, e^{cn}\}$, we have $g \le f$, so that $W(g) \subset W(f)$ and it is enough to prove that $E_W(g) > 0$.

If $K \in \N$ is such that
\begin{equation} \label{K}
\frac {\log (n+1)} n < c
\end{equation}
 for any integer $n$ greater or equal to $\lceil \frac K2 \rceil$ and $\sigma$ the morphism from
$\{a, b\}$ to $\{0, 1\}^{K+1}$ defined by $\sigma (a) = 0^{K+1}$ and $\sigma (b) = 0^{K}1$, we consider $w = \sigma  (C_{\{a, b\}}) \in \{0, 1\}^\N$ the image by $\sigma$ of the Champernowne word on the
alphabet $\{a, b\}$ (see Example \ref{example Ch})
 $$w = \sigma (C_{\{a, b\}}) = \sigma (abbabbbaababbbabbbbaaabaabbabababbbbaabbbabbbbabbbb\dots)$$
 $$ = 0^{2K+1}10^{K}10^{2K+1}10^{K}10^{K}10^{3K+2}10^{2K+1}10^{K}10^{K}10^{2K+1}10^{K}10^{K}10^{K}10^{4K+3}10^{3K+2}10^{K}\dots .$$
For any $n \in \{0, \dots , K+1 \}$ we have $p_w(n) = n+1$ and for any positive integer $n$ we have $p_w((K+1)n) \ge 2^n$ which implies that $$E(w) \ge \frac {\log 2} {K+1} > 0.$$
Let us now show that $w \in W(g),$ i. e. that for any non-negative integer $n$, we have $p_w(n) \le g(n)$:

- if $n \le K+1$, we have $p_w(n) = n+1 \le g(n)$;

- if $n \ge K+1$, we write $n = (K+1)Q + R$ with $Q \in \N$ and $R \in \{0, \dots, K\}$, so that if $R \ge \lceil \frac K2 \rceil,$ we have by (\ref{p_w}) and (\ref{K}) 
$$p_w(n) \le {(p_w(K+1))}^Q p_w(R)  \le e^{c(K+1)Q}  (R+1)\le e^{c(K+1)Q} e^{cR} = e^{cn}$$
and if $R < \lceil \frac K2 \rceil,$ writing $n = (K+1)(Q-1) + \lceil \frac K2 \rceil + (K+1+R - \lceil \frac K2 \rceil)$ and noticing that $(K+1+R - \lceil \frac K2 \rceil) \in \{\lceil \frac K2 \rceil, \dots , K+1 \}$ , we have again by (\ref{p_w}) and (\ref{K}) 
$$p_w(n) \le {(p_w(K+1))}^{(Q-1)} p_w(\lceil \frac K2 \rceil) p_w(K+1+R-\lceil \frac K2 \rceil)$$
$$ \le e^{c(K+1)(Q-1)} e^{c\lceil \frac K2 \rceil} e^{c(K+1+R-\lceil \frac K2 \rceil)} = e^{cn},$$
which ends the proof of Proposition \ref{corollary2}.

\end{proof}

\section{Examples of function $f$ such that $E_W(f) < E_0(f)$}  \label {relations}

We saw in Section \ref{?} that if we have $E_0(f)>0$ then $E_W(f) > 0$, but it is easy to check that
the ratio
$E_W(f)/E_0(f)$ can be arbitraly small.
Indeed, if $f$ is the function defined in Example \ref{example C}, we have $E_0(f) = \log \theta$.
But if $w \in W(f)$, we have  $p_w(n) = n+q-1$ for $1 \le n \le n_0$, which implies by (\ref{p_w}) that
$$p_w(n) \le p_w(n_0 \lceil n/n_0 \rceil) \le (p_w(n_0))^{ \lceil n/n_0 \rceil}   \le (n_0+q-1)^{\lceil
n/n_0 \rceil}$$ for any $n > n_0$
and it follows that $W(f)$ is contained in $W(g)$, where
$g(n) = n+q-1$ for $1 \le n \le n_0$ and $g(n)=(n_0+q-1)^{\lceil
n/n_0 \rceil}$ for $n > n_0$.
This shows, by Theorem \ref {prop3.0} and Claim \ref {prop2.1}, that we have
$$E_W(f) = \lim_{n\to+\infty}\frac 1n\log (|\sL_n(f)|)\le\lim_{n\to+\infty}\frac 1n\log (|\sL_n(g)|)\le E_0(g) = \frac1{n_0} \log (n_0+q-1),$$
where $\frac1{n_0} \log (n_0+q-1)$ can be made
arbitrarily small, independently of $\theta$.

\vskip .3in



\subsection{The conditions $(\mathcal C^*)$}

We will now suppose that functions $f$ satisfy the quite natural conditions $(\mathcal C^*)$
which hold for all unbounded complexity functions.

\begin{definition}\label{defC^*}

We say that a function $f$ from $\N$ to $\R^{+}$ satisfies the conditions $(\mathcal C^*)$ if

i) for any $n \in \N$ we have $f(n+1) > f(n) \ge n+1$ ;

ii) for any $(n, n' ) \in \N^2$ we have $f(n+n') \le f(n) f(n') $.

\end{definition}

\begin{remark}
If there exists $n \in \N$ such that $f(n) \le n$, then any $w \in A^{\N}$ such that $p_w \le f$ is ultimately periodic so that $W(f)$ is finite.
\end{remark}

\begin{remark}\label{rmk5.4}
Given any function $f$ from $\N$ to $\R^{+}$ such that $f(n)  \ge n+1$ for any $n \in \N$,
it is possible to construct recursively a non decreasing integer valued function $\tilde f$ satisfying  the condition $(\mathcal C^* (ii))$
and a real valued function 
$\tilde { \tilde f}$ satisfying  the conditions $(\mathcal C^*)$ such that
$W(f) = W(\tilde f) = W(\tilde { \tilde f})$ and such that $E_0(f) >0$ implies $E_0(\tilde f) = E_0(\tilde {\tilde f}) >0$.
For example, we can take $\tilde f$ and $\tilde { \tilde f}$ defined by $\tilde f (0) = 1$, $\tilde f(1) = \lfloor f(1) \rfloor$ and, for any integer $n \ge 2$,
$$\tilde f(n) = \min \{\inf_{\substack{ k \ge n}} \lfloor f(k) \rfloor, \min_{\substack{ 1 \le k < n}}  \tilde f (k) \tilde f (n-k) \}$$ and, for any $n \in \N$,
$$\tilde { \tilde f} (n) = \tilde f (n) + \frac n {n+1}.$$
\end{remark}

\begin{lemma}\label{rmkC^*}
If a function $f$ from $\N$ to $\R^{+}$ satisfies the
conditions $(\mathcal C^*)$ then, for any $ n \in \N$, we have $f(n) \ge \max \{n+1, \exp (E_0(f) n)  \}$.
\end{lemma}

\begin{proof}

This is a consequence of Lemma \ref {lemFekete} applied to $a_n = \log f(n)$.

\end{proof}

If $\theta ^{n_0} > n_0 + q - 1$, the example presented in the introduction of Section \ref{relations} does not satisfy the conditions $(\mathcal C^*)$ nor the weaker condition mentioned in Lemma \ref {rmkC^*}.
In Section  \ref{5.3} and Section \ref {cassaigne} we give two
examples which show that, even under these stronger conditions, we do
not always have $E_W(f)=E_0(f)$.

\subsection{An example of function $f$ satisfying $(\mathcal C^*)$ but such that $E_W(f) < E_0(f)$} \label{5.3}

Let $f$ be the function defined by 
$f(n)=\lceil 3^{n/2} \rceil$ for any $n \in \N$.
Then it is easy to check that $f$ satisfies conditions $(\mathcal C^*)$ and that
$E_0(f)=\lim\limits_{n\to\infty}\frac 1n \log f(n) =\log(\sqrt 3)$.

On the other hand, we have $f(1)=2$, so that $|A| = 2$. If $a \in A$, we denote $\bar a$ the other element of $A$.
We will show that the condition $f(2)=3$
implies that $E_W(f) \le \log (\frac{1+\sqrt 5}{2})<E_0(f)$. Indeed, given $w \in W(f)$, since $f(2)=3$, 
there is $(a_1, a_2) \in A^2$ such that $a_1 a_2 \notin L_2(w)$. It follows that for any $n \in \N$ we have
$$L_{n+2}(w) \subset L_{n+1}(w) a_2 \cup L_{n+1}(w) \bar a_2 = L_{n}(w) \bar a_1 a_2 \cup L_{n+1}(w)  \bar a_2,$$
which implies
$p_w(n+2) \le p_w(n+1)+p_w(n)$.
So it follows by induction that for any $n \in \N$ we have $p_w(n) \le F_{n+2}$,
where $(F_n)_{n\in\N}$ is the Fibonacci sequence, defined by $F_0=0, F_1=1$ and $F_{n+2} = F_{n+1}+F_n$ for any $n\in\N$.
It follows that $$E_W(f) \le \lim\limits_{n\to\infty}\frac 1n \log p_w(n) \le \lim\limits_{n\to\infty}\frac 1n \log F_{n+2} = \log (\frac{1+\sqrt 5}{2}).$$

For almost any infinite words $w$ on the alphabet $A=\{0,1\}$ without the factor $11$ we have
$p_w(n) = F_{n+2}$ for any $n \in \N$, so that we have proved that for this example, we have
$$E_W(f) = \log (\frac{1+\sqrt 5}{2})<\log(\sqrt 3) = E_0(f).$$

\subsection{Cassaigne conditions}

Cassaigne defined and studied in his thesis [Cas94] (see also [Cas96] and [Cas97]) the function $s_w$ counting (with multiplicity) the number of special factors 
of a given infinite word $w$.
More precisely if, for any $n \in \N$, we denote $$s_w(n) = \sum \limits_{v \in L_n(w)} (d^+(v)-1),$$
it follows from his study that $s_w(n)=p_w(n+1)-p_w(n)$ and that for any $(n, n') \in \N^2$ we have the general
inequality
$$s_w(n+n') \le p_w(n)s_w(n'),$$
which is equivalent to
$$p_w(n+n'+1)-p_w(n+n') \le p_w(n)(p_w(n'+1)-p_w(n')),$$
 which means that the sequence $(p_w(n)p_w(n') - p_w(n+n'))_{(n, n') \in \N^2}$ is non-decreasing in each variable.
In despite of that, we have the following example which shows that, even if we add to the conditions $(\mathcal C^*)$  the {\it Cassaigne condition}
\begin{equation} \label {cassaignecond}
f(n+n'+1)-f(n+n') \le f(n)(f(n'+1)-f(n')) \text{  for any } (n, n') \in \N^2,
\end{equation}
we may have $E_W(f)$ strictly smaller than $E_0(f)$.

\subsection{An example of function $f$ satisfying $(\mathcal C^*)$ and Cassaigne conditions but such that $E_W(f) < E_0(f)$} \label{cassaigne}

Let $f$ be given by $f(0)=1, f(1)=2, f(2)=4$ and $f(n)=f(n-1)+3f(n-3)$ for any $n \ge 3$. We claim that for any $k \in \N$, we have :

($H_k$)  For any $(n, n') \in \N^2$ such that $n' \le k$, we have  $f(n+n') \le f(n)f(n')$.\\
It is easy to check that $(H_1)$ is true by induction on $n$ : we have $f(n+1) \le 2 f(n)$ for $n \in \{0, 1, 2 \}$ and if we suppose that for
$n \ge 2$ we have $f(n-1) \le 2 f(n-2)$ and
$f(n+1) \le 2 f(n)$, then $f(n+2) = f(n+1) + 3 f(n-1) \le 2 f(n) + 6 f(n-2) = 2 f(n+1)$.
It follows from $(H_1)$  that $(H_2)$  is also true : $f(n+2) \le 4 f(n)$.

Then the proof of $(H_k)$  follows easily by induction on $k$ :
if we suppose that $(H_k)$ is true for $k-2$ and $k$ ($k \ge 2$), then for any
$n \in \N$ we have
$f(n+k+1) = f(n+k) + 3f(n+k-2) \le f(n)f(k) + 3f(n) f(k-2) = f(n)f(k+1),$
so that $(H_{k+1})$ is true.

As ($H_k$)  is true for any $k\in \N$, it follows that $f$ satisfies the conditions $(\mathcal C^*)$ and morevover satisfies the Cassaigne condition (\ref {cassaignecond}) which is 
equivalent in this example to the inequality $f(n+1) \le 2 f(n)$ for $n'=0$, to the inequality $f(n+2) - f(n+1) \le 2 f(n)$ for $n'=1$ and to the inequality
$f(n+n'-2) \le f(n)f(n'-2)$ for $n' \ge 2$.


We have $$E_0(f)=\lim\limits_{n\to\infty}\frac 1n \log f(n) = \log \al,$$
where $\al = \frac {\sqrt[3] {\frac {83}2 + \frac 92 \sqrt {85}} + \sqrt[3] {\frac {83}2 - \frac 92 \sqrt {85}}  +1} {3} \in (1,2)$ is the real root of 
the polynomial $p(x)=x^3-x^2-3$ (notice that $p(1)=-3<0$, $p(2)=1>0$ and $p$ is increasing on
$[1,+\infty)$, so that $1<\al<2$),

On the other hand, we have $f(1)=2$, so that $|A| = 2$. If $a \in A$, we denote again $\bar a$ the other element of $A$.
Given $w \in W(f)$, since $f(3)=7$, 
there is $(a_1, a_2, a_3) \in A^3$ such that $a_1 a_2 a_3\notin L_3(w)$. It follows that for any $n \in \N$ we have
\begin{align*}
L_{n+3}(w) &\subset L_{n+2}(w) a_3 \cup L_{n+2}(w) \bar a_3 \\
& \subset L_{n+1}(w) a_2 a_3 \cup  L_{n+1}(w) \bar a_2 a_3 \cup L_{n+2}(w) \bar a_3 \\
&= L_{n}(w) \bar a_1a_2 a_3 \cup  L_{n+1}(w) \bar a_2 a_3 \cup L_{n+2}(w) \bar a_3,
\end{align*}
which implies
$p_w(n+3) \le p_w(n+2)+p_w(n+1)+p_w(n)$.
So it follows by induction that for any $n \in \N$, we have $p_w(n) \le T_{n+2}$,
where $(T_n)_{n \in \N}$ is the Tribonacci sequence, defined by $T_0=0, T_1=1, T_2=1$ and $T_{n+3}=T_{n+2}+T_{n+1}+T_n$ for any $n \in \N$.
It follows that $$E_W(f) \le \lim\limits_{n\to\infty}\frac 1n \log p_w(n) \le \lim\limits_{n\to\infty}\frac 1n \log T_{n+2} = \log \beta,$$ where
$\beta = \frac {\sqrt[3] {19 + 3 \sqrt {33}} + \sqrt[3]{19 - 3 \sqrt {33}} +1} {3} \in (1,2)$ is the real root of the polynomial $q(x)=x^3-x^2-x-1$.
Since $p(\be)=\be^3-\be^2-3=\be^2+\be+1-\be^2-3=\be-2<0$, we have $\be < \al < 2$.

For almost any infinite words $w$ on the alphabet $A=\{0,1\}$ without the factor $111$ we have
$p_w(n) = T_{n+2}$ for any $n \in \N$, so that we have proved that for this example, we have
$$E_W(f) = \log \beta <\log \al = E_0(f).$$

\noindent
$\bold{Conclusion}$: Examples from sections \ref {5.3} and \ref{cassaigne} show the existence of functions $f$ satisfying the condition ($\mathcal C^*$) such that the ratio 
$E_W(f) / E_0(f)$ is equal respectively to $\log (\frac{1+\sqrt 5}{2}) / \log (\sqrt 3) = 0.876037...$ and $\log \beta / \log \al = 0.978814...$ and it would 
be interesting to determine how small can be this ratio when $f$ satisfies the condition ($\mathcal C^*$).

\begin{problem}
What can be said about $\inf \{ E_W(f) / E_0(f), f \mbox{  $satisfies$  } (\mathcal C^*) \}$ ?
\end{problem}

\section{Complexity functions of exponential order}

We will show that  for any  $h \in (0,+\infty)$, there is an infinite word $w$ such that $p_w$ is exactly of exponential order $h$. More precisely:
\begin{theorem} \label{6.10}
There exists a non-increasing function $C:(0,+\infty) \rightarrow [1,+\infty)$ such that, for any given $h \in (0,+\infty)$, there is an infinite word $w$ such that we have,
for any $n \in \N$, $e^{hn} \le p_w(n) \le C(h)e^{hn}$.
\end{theorem}

\begin{remark}
It will follow from the proof of Theorem \ref {6.10} that $C(h) \le 4$ for $h \ge \log 2$.
But the function $C(h)$ cannot be uniformly bounded.
Indeed if we have, for any $n \in \N$, $e^{hn} \le p_w(n) \le C(h)e^{hn}$, it follows (as $w$ in not ultimately periodic) that, for any $n \in \N$, we have
$n+1 \le p_w(n) \le C(h)e^{hn}$.
By choosing $n = \lfloor \frac 1h \rfloor$, it follows that $C(h) \ge (n+1){e^{-hn}} \ge \frac 1 {eh}$.

\end{remark}

\begin{proof}
If $q=\lceil e^h \rceil$ and $A=\{0,1,\dots,q-1\}$, we denote by $\le$
the lexicographic order on $A^\N$ (if $\al=a_1a_2a_3\dots$ and $\al '=a'_1a'_2a'_3\dots$, we put $\al \le \al '$ if $\al=\al '$ or if there is 
an integer $n \ge 1$ such that $a_k=a'_k$ for any $k<n$ and $a_n<a'_n$).
It follows from [Urb86] (or from [LM06]) that, for any $\al \in A^\N$, if we denote 
$$W_{\al}=\{\theta \in A^\N; T^n(\theta) \le \al, \forall n \in \N\} ,$$
$K_{\al} = \bold{w}^{-1}(W_{\al})$ ($w$ defined by (\ref {w}))
and $d(\al)$ the Hausdorff dimension of the set $K_{\al}$, then $$d(\al)=\frac{h_{\text top}(W_{\al}, T)}{\log q}$$ coincides with the upper box dimension of $K_{\al}$ and depends continuously on $\al$ (see Corollary 1 of Theorem 1 of [Gri]); 
moreover, the image of $d$ is [0,1].
Take $\al(h)$ minimum such that $d(\al(h))=h/\log q$ and notice that $\al(h)$ cannot end by $0^{\infty}$.
For each integer $n \ge 1$, let $\Lambda_n(h)=\{\ga_1^{(n)},\ga_2^{(n)},\dots,\ga_{p(n)}^{(n)}\}$ be the set of all factors of size $n$ of all words in $W_{\al(h)}$. 
We may choose positive integers $r_1^{(1)}, r_2^{(1)}, \dots, r_{p(1)}^{(1)}, r_1^{(2)}, r_2^{(2)}, \dots, r_1^{(3)}, r_2^{(3)}, \dots, r_{p(3)}^{(3)}, \dots$ such that the word 
$$w=\ga_1^{(1)}0^{r_1^{(1)}}\ga_2^{(1)}0^{r_2^{(1)}}\dots\ga_{p(1)}^{(1)}0^{r_{p(1)}^{(1)}}\ga_1^{(2)}0^{r_1^{(2)}}\ga_2^{(2)}0^{r_2^{(2)}}\dots\ga_{p(2)}^{(2)}0^{r_{p(2)}^{(2)}}\ga_1^{(3)}0^{r_1^{(3)}}\ga_2^{(3)}0^{r_2^{(3)}}\dots\ga_{p(3)}^{(3)}0^{r_{p(3)}^{(3)}}\dots$$
satisfies $L_n(w)=\Lambda_n(h)$ for any $n \ge 1$. 

If $q=2$, let $r \ge 1$ such that $\al(h)>10^{r-1}10^{\infty}$ and consider the injective map $g$ from $\Lambda_n(h)$ to $ A^{n+1}$ defined by
$$g(a_1 a_2 \dots a_n)= b_1 b_2 \dots b_{n+1},$$
where, for $i^*=\min\{i \mid a_j=0, \forall j, i<j \le n\}$, $b_{j+1}=a_j=0$ for $i^*<j\le n$, $b_{i^*+1}=a_{i^*}$, $b_{i^*}=0$ and $b_j=a_j$ for $j<i^*$.
(i.e. $g(a_1\dots a_{i^*-1} a_{i^*} 0 \dots  0) = a_1\dots a_{i^*-1}  0 a_{i^*} 0 \dots  0$).

Given $\ga \in \Lambda_n(h)$ and  $\ga' \in \Lambda_{n'}(h)$, we have $g(\ga) 0^r \ga' \in \Lambda_{n+n'+r+1}(h)$, so 
that for any $(n, n') \in \N^2$, we have $$2^{r+1}p_w(n+n')\ge p_w(n+n'+r+1) \ge p_w(n)p_w(n'),$$
and so $2^{-(r+1)}p_w(n+n') \ge (2^{-(r+1)}p_w(n))(2^{-(r+1)}p_w(n')).$
Lemma \ref {lemFekete} applied to $a_n = \frac {2^{r+1}} {p_w(n)}$ implies that, for any positive integer $n$, we have
$$2^{-(r+1)}p_w(n) \le \exp(n \lim_{n'\to+\infty}\frac 1{n'}\log(2^{-(r+1)}p_w(n'))) = \exp(hn),$$ so that $e^{hn}\le p_w(n)\le 2^{r+1} e^{hn}$ for any $n \in \N$.
Notice that, if $q=2$, $r$ is bounded if and only if $h$ is bounded away from $0$. 

If $q \ge 3$, consider for $1\le r \le q-2$ the injective maps $g_r$ from $\Lambda_n(h)$ to $A^{n+1}$ defined by
$g_r(0^{n})= 0^{n + 1}$ and, otherwise, $$g_r(a_1 a_2 \dots a_n)=b_1 b_2 \dots b_{n+1},$$
where, for $i^*=\min\{i \mid a_j=0, \forall j, i<j \le n\}$, $b_{j+1}=a_j=0$ for $i^*<j\le n$, $b_{i^*+1}=r$, $b_{i^*}=a_{i^*}-1$ and $b_j=a_j$ for $j<i^*$. Notice that the only point in the intersection of the images of two such maps is $0^{n}$.

Given $\ga \in \Lambda_n(h)$ and  $\ga' \in \Lambda_{n'}(h)$, we have $g_r(\ga) \ga' \in \Lambda_{n+n'+1}(h)$.
This implies that, for any $(n, n') \in \N^2$, we have
$$q \cdot p_w(n+n')\ge p_w(n+n'+1) \ge ((q-2)p_w(n)-(q-3))p_w(n'),$$
and, defining $\tilde p(n)=\frac{q-2}q p_w(n)-\frac{q-3}q$, we have for any $(n, n') \in \N^2$ with $n \ge 2$
$$\tilde p(n+n') \ge \tilde p(n) \tilde p(n')$$
(by symmetry it also holds when $n' \ge 2$) so that, as before, we have for any integer $n \ge 2$
$$e^{hn}\le p_w(n)=\frac q{q-2} \tilde p(n)+\frac{q-3}{q-2}\le \frac q{q-2} e^{hn}+\frac{q-3}{q-2} \le \frac {q+1}{q-2} e^{hn}.$$
Since we also have $p(1)=q \le \frac {(q+1)(q-1)}{q-2} \le \frac {q+1}{q-2} e^{h}$, we conclude that we have for any integer $n \ge 1$
$$e^{hn}\le p_w(n)\le \frac {q+1}{q-2} e^{hn} \le 4 e^{hn}.$$
\end{proof}

The following corollary arises immediately from Theorem \ref {6.10}
by considering, for any $h \in (0,+\infty)$,  the function defined by $f(n) = \max \{n+1, C(h)e^{hn} \}$ for any $n \in \N$:

\begin{corollary} \label{corollary}
For any $h \in (0,+\infty)$, there exists a function $f$ from $\N$ to $\R^{+}$ of exponential order $h$ satisfying conditions $(\mathcal C^*)$ such that $E_0(f) = E_W(f) = h$.
\end{corollary}

\section{An open problem} \label {open}

It is easy to see that, for any functions $f$ and $g$ from $\N$ to $\R^{+}$, we always have
$$E_W(\min\{f, g\})\le \min\{E_W(f), E_W(g)\}.$$
But we were not able to answer the following question:

\begin{problem} \label{problem}
Is it true that, for any functions $f$ and $g$ from $\N$ to $\R^{+}$ satisfying conditions $(\mathcal C^*)$ we have $E_W(\min\{f, g\})=\min\{E_W(f), E_W(g)\}$ ? 
\end{problem}

\begin{proposition} \label{equivalence}
The following two propositions are equivalent :

$(P_1)$ for any functions $f$ and $g$ from $\N$ to $\R^{+}$ satisfying the conditions $(\mathcal C^*)$ we have $E_W(\min\{f, g\}) = \min\{E_W(f), E_W(g)\}$;

$(P_2)$ for any $h\in (0,+\infty)$, there is a function $f_h$ satisfying the conditions $(\mathcal C^*)$ with $E_W(f_h)=h$ such that for any function $f$ with
$E_W(f)\ge h$ we have $f(n)\ge f_h(n)$ for any positive integer $n$.

Moreover, if $(P_1)$ or $(P_2)$ is true, there is a word $w$ for which $p_w(n)=f_h(n)$ for every positive integer $n$.

\end{proposition}

In order to prove Proposition \ref {equivalence}, we begin by giving an effective version of Lemma \ref{lem} :

\begin{lemma} \label{effective}
For any $c\in (0,\log q]$, we define the finctions $g_{(q,c)}$ and $f_{(q,c)}$ from ${\mathbb N}^*$ to ${\mathbb N}$
by $$g_{(q,c)}(n)=n\cdot \lceil \frac{n\log 2+e^{cn}(1+n(\log q-c))}{cn-\log(\lceil e^{cn} \rceil-1)} \rceil$$
and $$f_{(q,c)}(N)=\max_{\substack {1\le n\le N}} g_{(q,c)}(n) .$$  
Let $w\in \{0, 1,\dots,q-1\}^{\mathbb N}$ be an infinite word such that $E(w)\ge c$. Then, for any positive integer $N$, there is a factor $\gamma$  of $w$ of size
$f_{(q,c)}(N)$ such that, for every $n$ with $1\le n\le N$, the prefix of lenght $g_{(q,c)}(n)$ of $\gamma$ contains at least $\lceil e^{cn}\rceil$ factors of size $n$.

\end{lemma}   

\begin{proof}
Consider  $x \in [0,1]$ such that $\bold{w}(x)  = w$ and $X$  the closure of the orbit of $x$ by $\psi$, the shift map on $[0,1]$ defined in Section \ref {4.1}.
It is a compact set invariant by $\psi$, which has box dimension at least $c/\log q$. 

Fix an integer $n$ with $1\le n\le N$. The number of words $\gamma$ of size $g_{(q,c)}(n)$ with less than $\lceil e^{cn}\rceil$ factors of size $n$ is at most 
$${q^n \choose \lceil e^{cn}\rceil-1}(\lceil e^{cn}\rceil-1)^{g_{(q,c)}(n)/n}\le \left(\frac{e\cdot q^n}{e^{cn}}\right)^{e^{cn}}(\lceil e^{cn}\rceil-1)^{g_{(q,c)}(n)/n}.$$
Multiplying this estimate by $(1/q^{g_{(q,c)}(n)})^{c/\log q}=e^{-c g_{(q,c)}(n)}$, we get
$$\left(\frac{e\cdot q^n}{e^{cn}}\right)^{e^{cn}}(\lceil e^{cn}\rceil-1)^{g_{(q,c)}(n)/n}e^{-c g_{(q,c)}(n)}=$$
$$=\exp\left((1+n(\log q-c))e^{cn}-\frac{g_{(q,c)}(n)}n (cn-\log(\lceil e^{cn}\rceil-1))\right)\le$$
$$\le\exp\left((1+n(\log q-c))e^{cn}-(n\log 2+e^{cn}(1+n(\log q-c)))\right)=\exp(-n\log 2)=\frac1{2^n}.$$
If for any factor $\gamma$ of size $f_{(q,c)}(N)$ of $w$ there was an integer $n$ with $1\le n\le N$ such that the prefix of length
$g_{(q,c)}(n)$ of $\gamma$, contains less than $\lceil c^n\rceil$ factors of size $n$,
it would be possible to cover $X$ by a finite number of block intervals with disjoint interiors $I_j, 1 \le j \le m$ of sizes $1/q^{r_j}, 1 \le j \le m$ with
$$\sum_{j=1}^m (1/q^{r_j})^{c/\log q}=\sum_{n=1}^N k_n e^{-c g_{(q,c)}(n)}\le \sum_{n=1}^N\frac1{2^n}<1$$ and this, as in lemma \ref{lemdim},
would imply that the limit capacity of $X$ is strictly smaller than $c/\log q$ which would be a contradiction.  
\end{proof}

\begin{remark}
We may adapt the above proof in order to show the following variation of the previous proposition: in the same setting, define the finction $\tilde f_{(q,c)}$ from ${\mathbb N}^*$ to ${\mathbb N}$ by 
$$\tilde f_{(q,c)}(n)=n\cdot \lceil \frac{n\log 2+e^{cn}(1+n(\log q-c))}{\log 2} \rceil$$
and consider $w\in \{0, 1,\dots,q-1\}^{\mathbb N}$ an infinite word such that $E(w)\ge c$. Then, for any positive integer $N$, there is a factor
$\gamma$ of size $\tilde f_{(q,c)}(N)$ of $w$ such that, for any integer $n$ with $1\le n\le N$, the prefix of length $\tilde f_{(q,c)}(n)$ of $\gamma$
contains at least $\lceil \frac{e^{cn}}2\rceil$ factors of size $n$.
The conclusion of this statement is slightly weaker, but the function $\tilde f_{(q,c)}$ (which depends continuously on $c$)
 is more regular than $g_{(q,c)}$, which is not even locally bounded on $c$ if $e^{cn}\in \mathbb N$.
\end{remark} 

We can now prove Proposition \ref {equivalence}:

Let us first prove that $(P_2)$ implies $(P_1)$: if $\min\{E_W(f), E_W(g)\}=h$, we have $E_W(f)\ge h$ and $E_W(g)\ge h$, so that we have
$f(n)\ge f_h(n)$ and $g(n)\ge f_h(n)$ for any positive integer $n$, and thus $\min\{f, g\}(n)\ge f_h(n)$ for any positive integer $n$.
In particular, we have $$E_W(\min\{f, g\})\ge E_W(f_h)=h=\min\{E_W(f), E_W(g)\}.$$
Since $E_W(\min\{f, g\})\le \min\{E_W(f), E_W(g)\}$,
we have $E_W(\min\{f, g\})=\min\{E_W(f), E_W(g)\}$.

In order to prove that $(P_1)$ implies $(P_2)$, let us define the function $f_h$ by $$f_h(n)=\min\{p_w(n)|E(w)\ge h\}.$$
Given a positive integer $n$, let $w_n$ be an infinite word such that $E(w_n)\ge h$ and $p_{w_n}(n)=f_h(n)$ and let $g_k=p_{w_k}$ for $k\ge 1$.
If $(P_1)$ is true, an easy induction gives $$E_W(\min\{g_1,g_2,\dots,g_m\})=\min\{E_W(g_1),E_W(g_2),\dots,E_W(g_m)\}\ge h.$$
So there is an infinite word $\tilde w_m\in W(\min\{g_1,g_2,\dots,g_m\})$ with $E(\tilde w_m)\ge h$. Since $q_m=\min\{g_1,g_2,\dots,g_m\}(1)$ is eventually constant, we can assume without loss of generality that it is constant, equal to $q$.
Then Lemma \ref{effective} implies that there is a factor $\gamma_m$ of size $f_{(q,h)}(m)$ of $\tilde w_m$ such that, for any integer $n$ with $1\le n\le m$,
the prefix of length $g_{(q,h)}(n)$ of $\gamma_m$ contains at least $\lceil e^{hn}\rceil$ factors of size $n$.
Let $\hat w_m$ be an iterate of $\tilde w_m$ which begins by the factor $\gamma_m$.
Then, by compacity, there is a subsequence of $(\hat w_m)_{m\ge 1}$ converging to an infinite word $\hat w$ and, for every positive integer $n$,
the factor of length $g_{(q,h)}(n)$ of $\hat w$ contains at least $\lceil e^{hn}\rceil$ factors of size $n$.
In particular we have $E(\hat w)\ge h$, so that we have $$\lceil e^{hm}\rceil \le p_{\hat w}(m)\le g_m(m)=f_h(m).$$
On the other hand, given a function $f$ with $E_W(f)\ge h$, since there is $w\in W(f)$ with $E(w)=E_W(f)\ge h$, we have $f(n)\ge p_w(n)\ge f_h(n)$
for any positive integer $n$. In particular, since $E_W(p_{\hat w})\ge h$, we have $$p_{\hat w}(n)\ge f_h(n)$$ for any positive integer $n$, so that $p_{\hat w}(n)=f_h(n)$
for any positive integer $n$.
Now it follows from Theorem \ref{6.10} that for any positive integer $m$ we have $f_h(m)\le C(h)e^{hm}$, so that we have $E_W(f_h)=h$.

\end{document}